\documentclass[11pt,reqno]{amsart}
\usepackage{amsfonts,amsmath,amssymb,amsbsy,amstext,palatino, euler}
\usepackage{accents,color,enumerate,float,verbatim, mathbbol, bbm}
\usepackage{graphicx,epstopdf}
\usepackage[margin=1.1in,headheight=13pt]{geometry}
\usepackage[utf8]{inputenc}
\usepackage{indentfirst}
\usepackage{setspace, graphics}
\usepackage{mathtools}
\usepackage{xfakebold}
\usepackage{caption}  
\usepackage{tikz}
\usepackage{scalerel}
\usepackage{pict2e}
\usepackage{tkz-euclide}
\usetikzlibrary{calc}
\usetikzlibrary{external}
\usepackage{pgfplots}
\pgfplotsset{compat=newest}
\usepgfplotslibrary{statistics}
\usepgfplotslibrary{fillbetween}
\usepackage{adjustbox} 
\usepackage{comment}
\usepackage{microtype}
\usetikzlibrary{trees,positioning,arrows,chains,shapes.geometric,%
decorations.pathreplacing,decorations.pathmorphing,shapes,%
matrix,shapes.symbols,plotmarks,decorations.markings,shadows}

\newcommand{\fbseries}{\unskip\setBold\aftergroup\unsetBold\aftergroup\ignorespaces}
\makeatletter
\newcommand{\setBoldness}[1]{\def\fake@bold{#1}}
\makeatother
\newcommand{\contacts}{
 \footnotesize
  \bigskip
  \footnotesize

  A.~Asthana, \textsc{Dulles High School,
    Sugar Land, Texas}\par\nopagebreak
  \textit{E-mail address}: \texttt{stemanant@gmail.com}

  \medskip

  S.~Goyal, \textsc{McNeil High School, Austin, Texas}\par\nopagebreak
  \textit{E-mail address}: \texttt{shreev.goyal@gmail.com}

  \medskip

  A.~Dochtermann, \textsc{Department of Mathematics,
    Texas State University, San Marcos, Texas}\par\nopagebreak
  \textit{E-mail address}: \texttt{dochtermann@txstate.edu}

}

    \newtheorem{thm}{Theorem}[subsection]

    \newtheorem{prop}{Proposition}[section]
    \newtheorem{conj}{Conjecture}[section]
    \newtheorem{question}[thm]{Question}

    \theoremstyle{remark}
    \newtheorem{defn}{Definition}[section]

    \title{Rainbow Stars and Rota's Basis Conjecture for Graphic Matroids}

\date{October 2023}

\begin{document}

\maketitle

\begin{center}

\textsc{Anant Asthana, Shreev Goyal}

\vspace{1.5mm}

\textsc{\footnotesize Mentored by Dr. Anton Dochtermann}

\end{center}

\begin{abstract}
Let $G$ be a connected multigraph with $n$ vertices, and suppose $G$ has been edge-colored with $n-1$ colors so that each color class induces a spanning tree. Rota's Basis Conjecture for graphic matroids posits that one can find $n-1$ mutually edge-disjoint rainbow spanning trees. In a recent paper, Maezawa and Yazawa have shown that the conjecture holds if one assumes that the color classes induce spanning stars. We delve further into the star case to explore some extreme subcases including: all stars with different centers, the same center, or one of two centers. In addition, we identify the cases in which a graph composed of monochromatic stars can be decomposed into rainbow stars. We also show that the statement is false if one replaces `stars' with `paths'. 
\end{abstract}

\section{Introduction}
Rota's basis conjecture is a well-known open question in matroid theory that involves rearranging elements in a given set of bases to produce other `rainbow' bases. The statement of the conjecture first appeared in a paper of Huang and Rota \cite{HuaRot}. For the case of graphic matroids (which is our main focus here), the conjecture can be stated as follows.

\begin{conj}\label{conj:main}
 Let $G$ be a connected graph on $n$ vertices, and suppose $G$ has been edge colored with $n-1$ colors so that each color class induces a monochromatic spanning tree. Then $G$ can be decomposed into $n-1$ disjoint rainbow spanning trees.
 \end{conj}

We refer to Section \ref{sec:prelim} for any undefined terms. Note that Rota's Basis Conjecture is vacuous if the underlying graph $G$ is simple (the relevant number of disjoint spanning trees cannot exist), so here we are always working with graphs that have parallel edges.

According to Davies and McDiarmid \cite{K4}, graphs with no $K_4$ minor satisfy the conjecture. In a recent preprint, Maezawa and Yazawa \cite{MaeYaz} established Rota's Basis Conjecture for graphic matroids in the special case that each color class induces a monochromatic \emph{spanning star}. They prove that under this assumption $G$ can be decomposed into $n-1$ disjoint rainbow trees. Even for this special case, their proof is rather long and technical.

In this note we study special cases of Conjecture \ref{conj:main} involving arrangements of monochromatic stars and how one can find resulting rainbow spanning stars or trees. In particular, we consider the case of all stars having different centers, and the case of all stars having their centers among at most two vertices. 

Our main result states that if the color classes induce monochromatic stars, then one can find a collection of disjoint rainbow \emph{stars} if and only if all the given stars share the same center or all have different centers. See Theorem \ref{thm:main}. We also consider the case of induced paths and show, via a counterexample, that a similar statement does not hold in this case. See Proposition \ref{prop:path} for details. 

Along the way, we also consider the question of \emph{how many} collections of rainbow trees (or stars) can be found in a given edge colored graph (Rota's Conjecture states that at least one exists). We see that in some cases only one collection of rainbow stars can be found (see Proposition \ref{prop:onlyone}), and in other cases we observe a connection to Latin squares (see Theorem \ref{thm:latin}).


\section{Preliminaries}\label{sec:prelim}

 For us, a graph $G$ consists of a vertex set $V$ and a multiset of edges $E \subset V \times V$ where $e_{vw} \in E$ implies $w \neq v$ and $e_{wv} \in E$. Hence, our graphs are undirected and have no loops, but may (and usually will) have multiple edges. Unless otherwise specified, our graphs will be connected. 

 A \emph{tree} in $G$ is a set of edges $\tau \subset E$ with the property that the subgraph induced by $\tau$ has no closed paths (so in particular $\tau$ has no multiple edges). A \emph{star} in $G$ is a tree $S \subset E$ with the property that the subgraph induced by $S$ has a single vertex (called an \emph{center}) that is incident to every edge in $S$. A tree $\tau$ of $G$ is said to be \emph{spanning} (in $G$) if every vertex $v \in V$ is incident to some edge of $\tau$. Note that if $G$ is connected and has $n$ vertices, then a tree $\tau$ is spanning if and only if it consists of $n-1$ edges.

We will be interested in \emph{edge-colorings} of a graph $G$, by which we mean a function $f:E \rightarrow [m]$ where $m$ is some positive integer and $[m] = \{1,2, \dots, m\}$. A subgraph $H \subset G$ is said to be \emph{monochromatic} if every edge is assigned the same value, and is \emph{rainbow} if all edges are assigned different values.

Rota's Basis Conjecture tells us that if we are given $n-1$ monochromatic spanning trees on $n$ vertices, we can always find a set of $n-1$ rainbow spanning trees. For instance, in the example below, a graph on $n=4$ vertices composed of the top 3 monochromatic trees can be decomposed into the bottom 3 rainbow trees. Thus, this graph satisfies Rota's Basis Conjecture. 

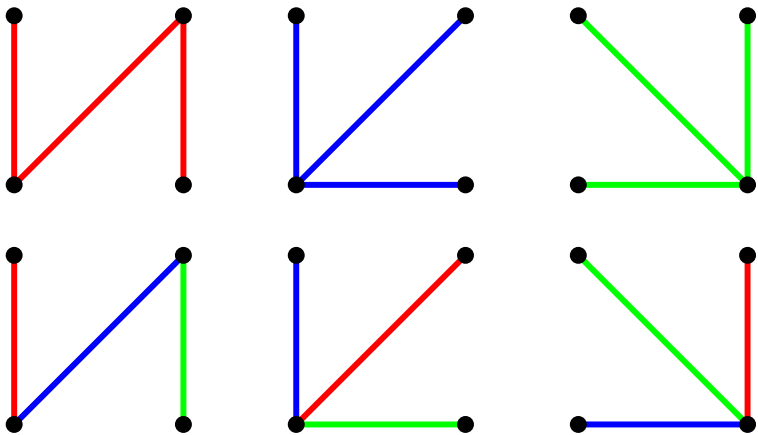
\vspace{2mm}
\begin{center} \begin{tikzpicture} [scale=0.75] 
\begin{adjustbox}{left} 
\draw[white] (0,0) grid (13,3); 
\draw[line width = 2.25, red] (0,0) -- (3,3); 
\draw[line width = 2.25, red] (0,3) -- (0,0); 
\draw[line width = 2.25, red] (3,3) -- (3,0); 
\fill[black] (0,3) circle [radius = 0.15]; 
\fill[black] (3,3) circle [radius = 0.15]; 
\fill[black] (0,0) circle [radius = 0.15]; 
\fill[black] (3,0) circle [radius = 0.15]; 
\draw[line width = 2.25, blue] (5,0) -- (8,0); 
\draw[line width = 2.25, blue] (5,0) -- (5,3); 
\draw[line width = 2.25, blue] (5,0) -- (8,3); 
\fill[black] (5,3) circle [radius = 0.15];
\fill[black] (8,3) circle [radius = 0.15]; 
\fill[black] (5,0) circle [radius = 0.15]; 
\fill[black] (8,0) circle [radius = 0.15]; 
\draw[line width = 2.25, green] (13,3) -- (13,0); 
\draw[line width = 2.25, green] (10,3) -- (13,0); 
\draw[line width = 2.25, green] (10,0) -- (13,0); 
\fill[black] (10,3) circle [radius = 0.15]; 
\fill[black] (13,3) circle [radius = 0.15]; 
\fill[black] (10,0) circle [radius = 0.15]; 
\fill[black] (13,0) circle [radius = 0.15]; 
\end{adjustbox}{left} \end{tikzpicture} 
\end{center} 

\vspace{5mm} 

\begin{center} \begin{tikzpicture} [scale=0.75] 
\begin{adjustbox}{left} 
\draw[white] (0,0) grid (13,3); 
\draw[line width = 2.25, blue] (0,0) -- (3,3); 
\draw[line width = 2.25, red] (0,3) -- (0,0); 
\draw[line width = 2.25, green] (3,3) -- (3,0); 
\fill[black] (0,3) circle [radius = 0.15]; 
\fill[black] (3,3) circle [radius = 0.15]; 
\fill[black] (0,0) circle [radius = 0.15]; 
\fill[black] (3,0) circle [radius = 0.15]; 
\draw[line width = 2.25, green] (5,0) -- (8,0); 
\draw[line width = 2.25, blue] (5,0) -- (5,3); 
\draw[line width = 2.25, red] (5,0) -- (8,3); 
\fill[black] (5,3) circle [radius = 0.15]; 
\fill[black] (8,3) circle [radius = 0.15]; 
\fill[black] (5,0) circle [radius = 0.15]; 
\fill[black] (8,0) circle [radius = 0.15]; 
\draw[line width = 2.25, red] (13,3) -- (13,0); 
\draw[line width = 2.25, green] (10,3) -- (13,0); 
\draw[line width = 2.25, blue] (10,0) -- (13,0); 
\fill[black] (10,3) circle [radius = 0.15]; 
\fill[black] (13,3) circle [radius = 0.15]; 
\fill[black] (10,0) circle [radius = 0.15]; 
\fill[black] (13,0) circle [radius = 0.15]; 
\end{adjustbox}{left} \end{tikzpicture} \end{center}
\vspace{2mm}

In fact, Rota's Basis Conjecture can be seen as the 'Trees-to-Trees' extension of 'Stars-to-Stars' that we make precise below.

\begin{defn}
Suppose $G$ is a graph on $n$ vertices that has been edge colored with $n-1$ colors so that each monochromatic subgraph induces a star. Then we say \textbf{Stars-to-stars}
 holds for $G$ if one can find a collection of $n-1$ disjoint rainbow stars.
\end{defn}

\section{Stars-to-Stars and Stars-to-Trees}
\subsection{All Stars Have a Different Center}

\begin{thm}\label{thm:starsdifferent}
    Stars-to-Stars holds if all monochromatic stars have a different center.
\end{thm}

\begin{proof}

Suppose $G$ is a connected graph on vertex set $V = \{v_1,\ldots,v_n\}$ and edge set $E$. Suppose $G$ has been edge colored with $n-1$ colors so that each color class induces a monochromatic star. Without loss of generality, suppose $v_1, \dots, v_{n-1}$ are the internal nodes of these stars. 


Now, we claim that this graph can be decomposed into $n-1$ rainbow stars.
Let $c_j$ denote the color of the monochromatic star with center at $v_j$; we will denote the whole monochromatic star as $S_j$. To construct $\textsc{RS}_j$, the rainbow star with center at vertex $j$, we do the following: 
\begin{enumerate}
    \item We choose the edge $e_{jn} \in S_j$ to connect node $v_n$ to $v_j$ such that $e_{jn}$ has color $c_j$.
    \item Then, for any node $v_i$ such that $i \neq n,j$, we choose the edge $e_{ji} \in S_i$ between $v_i$ and $v_j$ to connect node $i$ to $j$.
\end{enumerate}
Thus, $$\textsc{RS}_j = \{V \cup e_{ji} \mid i \neq n, i \neq j, e_{ji} \in S_i, \text{ and } e_{jn} \in S_j\}.$$ 

\vspace{4mm}

For instance, consider the following graph on $n=4$ vertices:

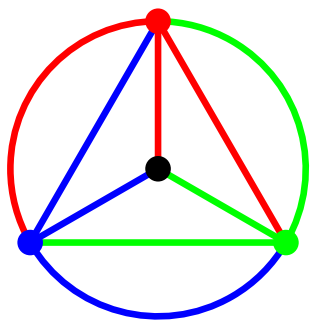
\begin{center}
\begin{tikzpicture}[scale=0.85] 

\draw[white] (0,0) grid (8,3);

\draw[line width = 2.5, red] (4,2/3*1.73205) -- (4,2*1.73205);
\draw[line width = 2.5, blue] (4,2/3*1.73205) -- (2,0);
\draw[line width = 2.5, green] (4,2/3*1.73205) -- (6,0);

\draw[line width = 2.5, red] (4,2*1.73205) -- (6,0);
\draw[line width = 2.5, blue] (2,0) -- (4,2*1.73205);
\draw[line width = 2.5, green] (6,0) -- (2,0);

\draw[line width = 2.5, red] (4,2*1.73205) arc (90:210:4/3*1.73205);
\draw[line width = 2.5, blue] (2,0) arc (210:330:4/3*1.73205);
\draw[line width = 2.5, green] (6,0) arc (-30:90:4/3*1.73205);

\fill[red] (4,2*1.73205) circle [radius = 0.2];
\fill[blue] (2,0) circle [radius = 0.2];
\fill[green] (6,0) circle [radius = 0.2];
\fill[black] (4,2/3*1.73205) circle [radius = 0.2];

\end{tikzpicture}
\end{center}

The nodes of this graph have already been rearranged as detailed above. The center vertex has also been colored black to distinguish it from the centers of the monochromatic spanning stars. Notice that this graph is induced by the following $3$ monochromatic stars:
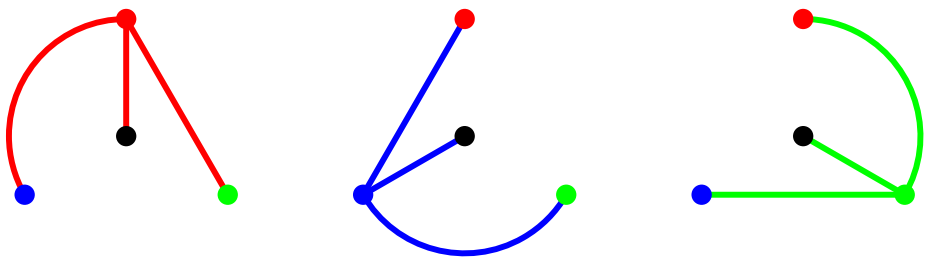
\vspace{2mm}

\begin{center}
\begin{tikzpicture}[scale=0.9] 

\draw[white] (0,0) grid (13,2.598076);

\draw[line width = 2.25, red] (1.5,1.5*1.73205) -- (1.5,.5*1.73205);
\draw[line width = 2.25, red] (1.5,1.5*1.73205) -- (3,0);
\draw[line width = 2.25, red] (1.5,1.5*1.73205) arc (90:210:1*1.73205);

\fill[red] (1.5,1.5*1.73205) circle [radius = 0.15];
\fill[blue] (0,0) circle [radius = 0.15];
\fill[green] (3,0) circle [radius = 0.15];
\fill[black] (1.5,.5*1.73205) circle [radius = 0.15];


\draw[line width = 2.25, blue] (5,0) -- (6.5,1.5*1.73205);
\draw[line width = 2.25, blue] (5,0) -- (6.5,.5*1.73205);
\draw[line width = 2.25, blue] (5,0) arc (210:330:1*1.73205);

\fill[red] (6.5,1.5*1.73205) circle [radius = 0.15];
\fill[blue] (5,0) circle [radius = 0.15];
\fill[green] (8,0) circle [radius = 0.15];
\fill[black] (6.5,.5*1.73205) circle [radius = 0.15];


\draw[line width = 2.25, green] (13,0) -- (10,0);
\draw[line width = 2.25, green] (13,0) -- (11.5,.5*1.73205);
\draw[line width = 2.25, green] (13,0) arc (-30:90:1*1.73205);

\fill[red] (11.5,1.5*1.73205) circle [radius = 0.15];
\fill[blue] (10,0) circle [radius = 0.15];
\fill[green] (13,0) circle [radius = 0.15];
\fill[black] (11.5,.5*1.73205) circle [radius = 0.15];

\end{tikzpicture}
\end{center}


In addition, the center of each star has been colored with the same color as the edges of that monochromatic star. We follow the algorithm outlined in the proof to dissect the construction into 3 rainbow stars, as shown below.






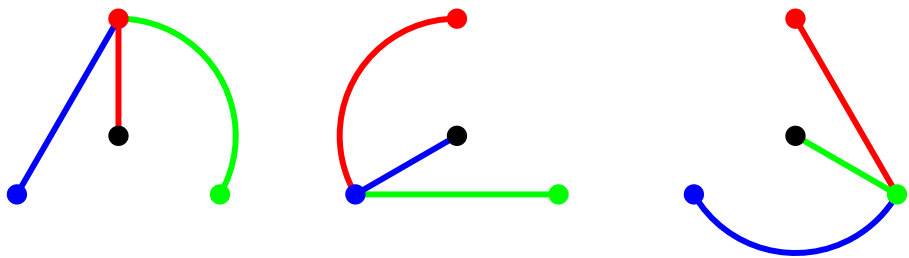
\begin{center}
\begin{tikzpicture}[scale=0.9] 

\draw[white] (0,0) grid (13,2.598076);

\draw[line width = 2.25, red] (1.5,1.5*1.73205) -- (1.5,.5*1.73205);
\draw[line width = 2.25, blue] (0,0) -- (1.5,1.5*1.73205);
\draw[line width = 2.25, green] (3,0) arc (-30:90:1*1.73205);

\fill[red] (1.5,1.5*1.73205) circle [radius = 0.15];
\fill[blue] (0,0) circle [radius = 0.15];
\fill[green] (3,0) circle [radius = 0.15];
\fill[black] (1.5,.5*1.73205) circle [radius = 0.15];


\draw[line width = 2.25, red] (6.5,1.5*1.73205) arc (90:210:1*1.73205);
\draw[line width = 2.25, green] (5,0) -- (8,0);
\draw[line width = 2.25, blue] (5,0) -- (6.5,.5*1.73205);

\fill[red] (6.5,1.5*1.73205) circle [radius = 0.15];
\fill[blue] (5,0) circle [radius = 0.15];
\fill[green] (8,0) circle [radius = 0.15];
\fill[black] (6.5,.5*1.73205) circle [radius = 0.15];


\draw[line width = 2.25, red] (11.5,1.5*1.73205) -- (13,0);
\draw[line width = 2.25, green] (13,0) -- (11.5,.5*1.73205);
\draw[line width = 2.25, blue] (10,0) arc (210:330:1*1.73205);

\fill[red] (11.5,1.5*1.73205) circle [radius = 0.15];
\fill[blue] (10,0) circle [radius = 0.15];
\fill[green] (13,0) circle [radius = 0.15];
\fill[black] (11.5,.5*1.73205) circle [radius = 0.15];

\end{tikzpicture}
\end{center}


Note that all of these rainbow stars are disjoint and use all the edges of the original graph. In other words, all the rainbow stars are built by connecting a colored point and the center point with its own color, and then receiving a color from every other colored point. We can guarantee that all edges are used exactly once. Thus, the proof is complete. 
\end{proof}

\subsection*{Example with 5 nodes}


The same process can be used to decompose a graph with $5$ vertices (induced by $4$ monochromatic spanning stars) to create $4$ rainbow stars, as depicted below.

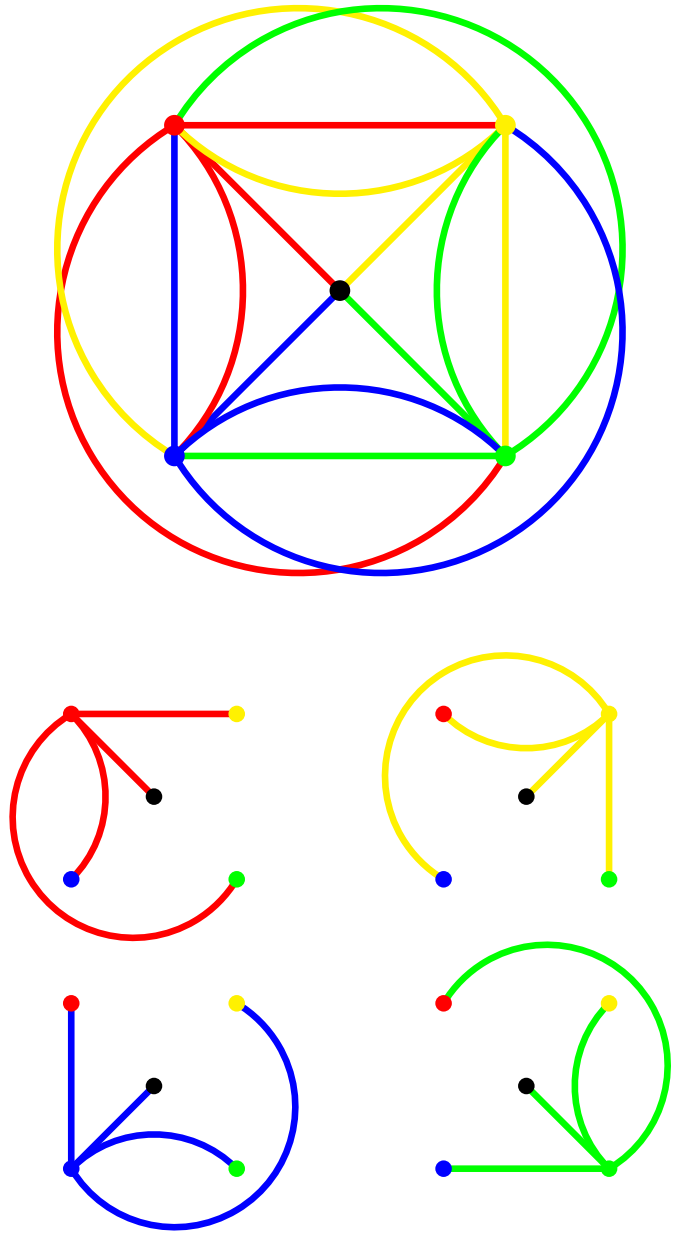
\begin{center}
\begin{tikzpicture}[scale=0.55] 
\begin{adjustbox}{left}

\draw[white] (0,0) grid (8,8);

\draw[line width = 2.5, red] (0,8) -- (8,8);
\draw[line width = 2.5, yellow] (8,8) -- (8,0);
\draw[line width = 2.5, green] (8,0) -- (0,0);
\draw[line width = 2.5, blue] (0,0) -- (0,8);

\draw[line width = 2.5, red] (0,8) -- (4,4);
\draw[line width = 2.5, yellow] (8,8) -- (4,4);
\draw[line width = 2.5, green] (8,0) -- (4,4);
\draw[line width = 2.5, blue] (0,0) -- (4,4);

\draw[line width = 2.5, red] (0,8) arc (90+30.96376:360-30.96376:5.83095);
\draw[line width = 2.5, yellow] (8,8) arc (30.96376:270-30.96376:5.83095);
\draw[line width = 2.5, green] (8,0) arc (-90+30.96376:180-30.96376:5.83095);
\draw[line width = 2.5, blue] (0,0) arc (-180+30.96376:90-30.96376:5.83095);

\draw[line width = 2.5, red] (0,8) arc (45:-45:5.65685);
\draw[line width = 2.5, yellow] (8,8) arc (-45:-135:5.65685);
\draw[line width = 2.5, green] (8,0) arc (225:135:5.65685);
\draw[line width = 2.5, blue] (0,0) arc (135:45:5.65685);

\fill[red] (0,8) circle [radius = 0.25];
\fill[blue] (0,0) circle [radius = 0.25];
\fill[green] (8,0) circle [radius = 0.25];
\fill[yellow] (8,8) circle [radius = 0.25];
\fill[black] (4,4) circle [radius = 0.25];

\end{adjustbox}
\end{tikzpicture}
\end{center}



\begin{center}
\begin{tikzpicture}[scale=0.55] 
\begin{adjustbox}{left}

\draw[white] (0,0) grid (13,11);

\draw[line width = 2.5, red] (0,11) -- (4,11);
\draw[line width = 2.5, red] (0,11) -- (2,9);
\draw[line width = 2.5, red] (0,11) arc (90+30.96376:360-30.96376:5.83095/2);
\draw[line width = 2.5, red] (0,11) arc (45:-45:5.65685/2);

\draw[line width = 2.5, yellow] (13,11) -- (13,7);
\draw[line width = 2.5, yellow] (13,11) -- (11,9);
\draw[line width = 2.5, yellow] (13,11) arc (30.96376:270-30.96376:5.83095/2);
\draw[line width = 2.5, yellow] (13,11) arc (-45:-135:5.65685/2);

\draw[line width = 2.5, blue] (0,0) -- (0,4);
\draw[line width = 2.5, blue] (0,0) -- (2,2);
\draw[line width = 2.5, blue] (0,0) arc (-180+30.96376:90-30.96376:5.83095/2);
\draw[line width = 2.5, blue] (0,0) arc (135:45:5.65685/2);

\draw[line width = 2.5, green] (13,0) -- (9,0);
\draw[line width = 2.5, green] (13,0) -- (11,2);
\draw[line width = 2.5, green] (13,0) arc (-90+30.96376:180-30.96376:5.83095/2);
\draw[line width = 2.5, green] (13,0) arc (225:135:5.65685/2);

\fill[red] (0,11) circle [radius = 0.2];
\fill[blue] (0,7) circle [radius = 0.2];
\fill[green] (4,7) circle [radius = 0.2];
\fill[yellow] (4,11) circle [radius = 0.2];
\fill[black] (2,9) circle [radius = 0.2];

\fill[red] (9,11) circle [radius = 0.2];
\fill[blue] (9,7) circle [radius = 0.2];
\fill[green] (13,7) circle [radius = 0.2];
\fill[yellow] (13,11) circle [radius = 0.2];
\fill[black] (11,9) circle [radius = 0.2];

\fill[red] (0,4) circle [radius = 0.2];
\fill[blue] (0,0) circle [radius = 0.2];
\fill[green] (4,0) circle [radius = 0.2];
\fill[yellow] (4,4) circle [radius = 0.2];
\fill[black] (2,2) circle [radius = 0.2];

\fill[red] (9,4) circle [radius = 0.2];
\fill[blue] (9,0) circle [radius = 0.2];
\fill[green] (13,0) circle [radius = 0.2];
\fill[yellow] (13,4) circle [radius = 0.2];
\fill[black] (11,2) circle [radius = 0.2];

\end{adjustbox}
\end{tikzpicture}
\end{center}

Shown above are the $4$ monochromatic stars that comprise the graph in the previous figure. 

Following the algorithm described in the previous proof, we connect each center of a $c$ color-induced star to the center vertex, and we connect that same center to the remaining vertices using the respective colors of the stars whose center is at that remaining vertex. 

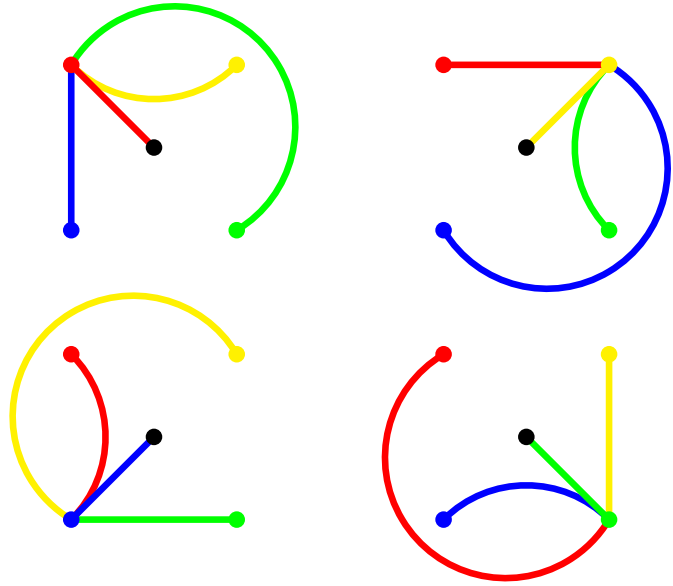
\begin{center}
\begin{tikzpicture}[scale=0.55] 
\begin{adjustbox}{left}

\draw[white] (0,0) grid (13,11);

\draw[line width = 2.5, blue] (0,7) -- (0,11);
\draw[line width = 2.5, yellow] (4,11) arc (-45:-135:5.65685/2);
\draw[line width = 2.5, green] (4,7) arc (-90+30.96376:180-30.96376:5.83095/2);
\draw[line width = 2.5, red] (0,11) -- (2,9);

\draw[line width = 2.5, red] (9,11) -- (13,11);
\draw[line width = 2.5, blue] (9,7) arc (-180+30.96376:90-30.96376:5.83095/2);
\draw[line width = 2.5, green] (13,7) arc (225:135:5.65685/2);
\draw[line width = 2.5, yellow] (13,11) -- (11,9);

\draw[line width = 2.5, green] (0,0) -- (4,0);
\draw[line width = 2.5, yellow] (4,4) arc (30.96376:270-30.96376:5.83095/2);
\draw[line width = 2.5, red] (0,4) arc (45:-45:5.65685/2);
\draw[line width = 2.5, blue] (0,0) -- (2,2);

\draw[line width = 2.5, yellow] (13,4) -- (13,0);
\draw[line width = 2.5, red] (9,4) arc (90+30.96376:360-30.96376:5.83095/2);
\draw[line width = 2.5, blue] (9,0) arc (135:45:5.65685/2);
\draw[line width = 2.5, green] (13,0) -- (11,2);

\fill[red] (0,11) circle [radius = 0.2]; 
\fill[blue] (0,7) circle [radius = 0.2];
\fill[green] (4,7) circle [radius = 0.2];
\fill[yellow] (4,11) circle [radius = 0.2];
\fill[black] (2,9) circle [radius = 0.2];

\fill[red] (9,11) circle [radius = 0.2];
\fill[blue] (9,7) circle [radius = 0.2];
\fill[green] (13,7) circle [radius = 0.2];
\fill[yellow] (13,11) circle [radius = 0.2];
\fill[black] (11,9) circle [radius = 0.2]; 

\fill[red] (0,4) circle [radius = 0.2];
\fill[blue] (0,0) circle [radius = 0.2];
\fill[green] (4,0) circle [radius = 0.2];
\fill[yellow] (4,4) circle [radius = 0.2];
\fill[black] (2,2) circle [radius = 0.2];

\fill[red] (9,4) circle [radius = 0.2];
\fill[blue] (9,0) circle [radius = 0.2];
\fill[green] (13,0) circle [radius = 0.2];
\fill[yellow] (13,4) circle [radius = 0.2];
\fill[black] (11,2) circle [radius = 0.2];

\end{adjustbox}
\end{tikzpicture}
\end{center}


This algorithm can be used for an arbitrary number of points.

In the context of Rota's Conjecture, a natural question to ask is \emph{how many} sets of rainbow spanning objects one can find in certain contexts. In the case of finding rainbow stars from monochromatic stars, we have the following observation.

\begin{prop}\label{prop:onlyone}
Suppose $G$ is a graph on $n$ vertices that has been edge colored with $n-1$ colors so that the induced monochromatic subgraphs are all stars with different centers. Then there exists a unique collection of $n-1$ disjoint spanning rainbow stars.
\end{prop}

\begin{proof}
    If we take a closer look, we can identify a few points to consider: 
\begin{enumerate}[(a)]
\item The center node can not have a star because if it did, this would use up the edges connecting the center and the remaining trees could not connect the center.
\item Each of the colored stars on their nodes contain $n-1$ edges of its own color and 1 color edge from the other $n-2$ stars.;
\item Each star is created by combining the lone edge to the center plus the collection of edges from the other stars.
\end{enumerate}

This means there is one and exactly one way to create each star, so the total number of collections of disjoint spanning rainbow stars is one.
\end{proof}


\subsection{All Stars Have the Same Center}\label{sec:allsame}
\begin{thm}\label{thm:starssame}
    Stars-to-Stars holds if all monochromatic stars have the same center.
\end{thm}

\begin{proof}
Say a graph $G = (V,E)$ satisfies the conditions of the statement. We will construct the resulting rainbow stars. Let $v_i$ denote any given vertex of the graph, with $v_0$ signifying the common center of all the monochromatic stars. For $1 \le j \le n-1$, we define the $j$th resulting star $RS_j \subset E$ in the following manner:  
\[RS_j = \{e_{v_0v_{j+k-1}} \mid e_{v_0v_{j+k-1}} \in S_k, 1 \le k \le n-1\},\] where $j+k-1$ is taken mod $n-1$. 
\end{proof}

For instance, in the example below on 4 vertices and 3 colors, the following trees result. Notice how the colors are merely rotated between trees.

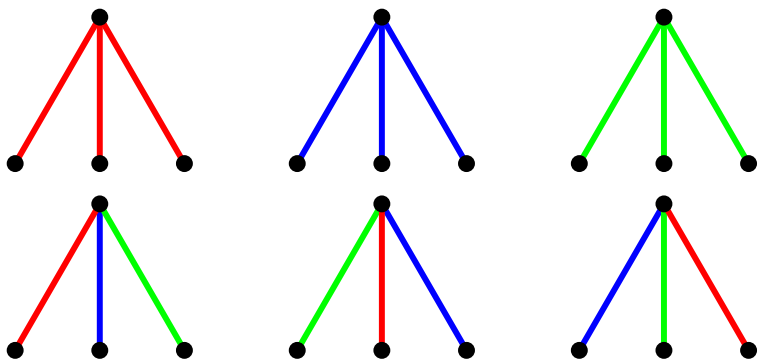
\begin{center}
\begin{tikzpicture}[scale=0.75] 
\begin{adjustbox}{left}
\draw[white] (0,0) grid (13,3);

\draw[line width = 2.25, red] (1.5,1.5*1.73205) -- (1.5,0);
\draw[line width = 2.25, red] (1.5,1.5*1.73205) -- (3,0);
\draw[line width = 2.25, red] (1.5,1.5*1.73205) -- (0,0);

\fill[black] (1.5,1.5*1.73205) circle [radius = 0.15];
\fill[black] (0,0) circle [radius = 0.15];
\fill[black] (3,0) circle [radius = 0.15];
\fill[black] (1.5,0) circle [radius = 0.15];


\draw[line width = 2.25, blue] (6.5,1.5*1.73205) -- (5,0);
\draw[line width = 2.25, blue] (6.5,1.5*1.73205) -- (8,0);
\draw[line width = 2.25, blue] (6.5,1.5*1.73205) -- (6.5,0);

\fill[black] (6.5,1.5*1.73205) circle [radius = 0.15];
\fill[black] (5,0) circle [radius = 0.15];
\fill[black] (8,0) circle [radius = 0.15];
\fill[black] (6.5,0) circle [radius = 0.15];


\draw[line width = 2.25, green] (11.5,1.5*1.73205) -- (10,0);
\draw[line width = 2.25, green] (11.5,1.5*1.73205) -- (11.5,0);
\draw[line width = 2.25, green] (11.5,1.5*1.73205) -- (13,0);

\fill[black] (11.5,1.5*1.73205) circle [radius = 0.15];
\fill[black] (10,0) circle [radius = 0.15];
\fill[black] (13,0) circle [radius = 0.15];
\fill[black] (11.5,0) circle [radius = 0.15];
\end{adjustbox}{left}
\end{tikzpicture}
\end{center}

\begin{center}
\begin{tikzpicture}[scale=0.75] 
\begin{adjustbox}{left}
\draw[white] (0,0) grid (13,3);

\draw[line width = 2.25, blue] (1.5,1.5*1.73205) -- (1.5,0);
\draw[line width = 2.25, green] (1.5,1.5*1.73205) -- (3,0);
\draw[line width = 2.25, red] (1.5,1.5*1.73205) -- (0,0);

\fill[black] (1.5,1.5*1.73205) circle [radius = 0.15];
\fill[black] (0,0) circle [radius = 0.15];
\fill[black] (3,0) circle [radius = 0.15];
\fill[black] (1.5,0) circle [radius = 0.15];


\draw[line width = 2.25, green] (6.5,1.5*1.73205) -- (5,0);
\draw[line width = 2.25, blue] (6.5,1.5*1.73205) -- (8,0);
\draw[line width = 2.25, red] (6.5,1.5*1.73205) -- (6.5,0);

\fill[black] (6.5,1.5*1.73205) circle [radius = 0.15];
\fill[black] (5,0) circle [radius = 0.15];
\fill[black] (8,0) circle [radius = 0.15];
\fill[black] (6.5,0) circle [radius = 0.15];


\draw[line width = 2.25, blue] (11.5,1.5*1.73205) -- (10,0);
\draw[line width = 2.25, green] (11.5,1.5*1.73205) -- (11.5,0);
\draw[line width = 2.25, red] (11.5,1.5*1.73205) -- (13,0);

\fill[black] (11.5,1.5*1.73205) circle [radius = 0.15];
\fill[black] (10,0) circle [radius = 0.15];
\fill[black] (13,0) circle [radius = 0.15];
\fill[black] (11.5,0) circle [radius = 0.15];
\end{adjustbox}{left}
\end{tikzpicture}
\end{center}

Once again we are interested in how many sets of rainbow spanning trees can be found in this context, and in particular we consider the following.

\begin{question}
    From a configuration with all monochromatic stars sharing the same center, how many collections of $n-1$ spanning rainbow stars are possible?
\end{question}
Returning to the $n=4$ example where there are 3 rainbow stars, there exist exactly 2 distinct unordered collections, as illustrated below.
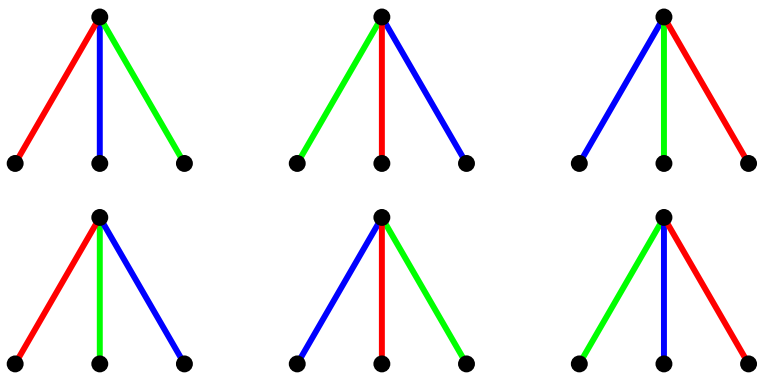
\vspace{2mm}

\begin{center}
\begin{tikzpicture}[scale=0.75] 
\begin{adjustbox}{left}
\draw[white] (0,0) grid (13,3);

\draw[line width = 2.25, blue] (1.5,1.5*1.73205) -- (1.5,0);
\draw[line width = 2.25, green] (1.5,1.5*1.73205) -- (3,0);
\draw[line width = 2.25, red] (1.5,1.5*1.73205) -- (0,0);

\fill[black] (1.5,1.5*1.73205) circle [radius = 0.15];
\fill[black] (0,0) circle [radius = 0.15];
\fill[black] (3,0) circle [radius = 0.15];
\fill[black] (1.5,0) circle [radius = 0.15];


\draw[line width = 2.25, green] (6.5,1.5*1.73205) -- (5,0);
\draw[line width = 2.25, blue] (6.5,1.5*1.73205) -- (8,0);
\draw[line width = 2.25, red] (6.5,1.5*1.73205) -- (6.5,0);

\fill[black] (6.5,1.5*1.73205) circle [radius = 0.15];
\fill[black] (5,0) circle [radius = 0.15];
\fill[black] (8,0) circle [radius = 0.15];
\fill[black] (6.5,0) circle [radius = 0.15];


\draw[line width = 2.25, blue] (11.5,1.5*1.73205) -- (10,0);
\draw[line width = 2.25, green] (11.5,1.5*1.73205) -- (11.5,0);
\draw[line width = 2.25, red] (11.5,1.5*1.73205) -- (13,0);

\fill[black] (11.5,1.5*1.73205) circle [radius = 0.15];
\fill[black] (10,0) circle [radius = 0.15];
\fill[black] (13,0) circle [radius = 0.15];
\fill[black] (11.5,0) circle [radius = 0.15];
\end{adjustbox}{left}
\end{tikzpicture}
\end{center}
\vspace{1mm}
\begin{center}
\begin{tikzpicture}[scale=0.75] 
\begin{adjustbox}{left}
\draw[white] (0,0) grid (13,3);

\draw[line width = 2.25, green] (1.5,1.5*1.73205) -- (1.5,0);
\draw[line width = 2.25, blue] (1.5,1.5*1.73205) -- (3,0);
\draw[line width = 2.25, red] (1.5,1.5*1.73205) -- (0,0);

\fill[black] (1.5,1.5*1.73205) circle [radius = 0.15];
\fill[black] (0,0) circle [radius = 0.15];
\fill[black] (3,0) circle [radius = 0.15];
\fill[black] (1.5,0) circle [radius = 0.15];


\draw[line width = 2.25, blue] (6.5,1.5*1.73205) -- (5,0);
\draw[line width = 2.25, green] (6.5,1.5*1.73205) -- (8,0);
\draw[line width = 2.25, red] (6.5,1.5*1.73205) -- (6.5,0);

\fill[black] (6.5,1.5*1.73205) circle [radius = 0.15];
\fill[black] (5,0) circle [radius = 0.15];
\fill[black] (8,0) circle [radius = 0.15];
\fill[black] (6.5,0) circle [radius = 0.15];


\draw[line width = 2.25, green] (11.5,1.5*1.73205) -- (10,0);
\draw[line width = 2.25, blue] (11.5,1.5*1.73205) -- (11.5,0);
\draw[line width = 2.25, red] (11.5,1.5*1.73205) -- (13,0);

\fill[black] (11.5,1.5*1.73205) circle [radius = 0.15];
\fill[black] (10,0) circle [radius = 0.15];
\fill[black] (13,0) circle [radius = 0.15];
\fill[black] (11.5,0) circle [radius = 0.15];
\end{adjustbox}{left}
\end{tikzpicture}
\end{center}
\vspace{2mm}

For $n=2,3$, the number of collections is $1$. For $n=4$, the number of collections increases to $2$. For $n=5$, the number of collections jumps to 24. $n=6$ yields an even larger result that is much more difficult to calculate by hand (see table). 

In fact, we claim the following: 
\begin{thm} \label{thm:latin}
    For a graph on $n$ vertices where all induced spanning monochromatic stars share the same center, the number of collections of $n-1$ disjoint spanning rainbow stars is 
\[\Omega(n) = \frac{L_{n-1}}{(n-1)!},\]
\noindent    
where $L_{n-1}$ is the number of Latin squares of size $n-1$.    
\end{thm}

\begin{proof}

The strategy involves representing the set of all edges of all stars as a matrix.

Consider each resulting star as a row of $n-1$ elements, where the $i$th element in the row is the color of the edge connecting $v_0$ and $v_i$. Considering all the resulting rainbow stars at once, we can combine $n-1$ rows/stars of $n-1$ elements each into a \textbf{rainbow star matrix}. In other words, the $i$th star would be $i$th row in the matrix, and the $j$th edge in the $i$th star would be the $j$th element in the $i$th row. 

For instance, in the example shown below, the leftmost star would represent the row $\{R,B,G\}$, the next would be $\{G,R,B\}$, and the last would be $\{B,G,R\}$, where R=red, B=blue, and G=green. 

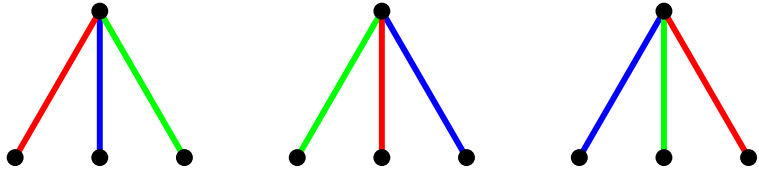
\begin{center}
\begin{tikzpicture}[scale=0.75] 
\begin{adjustbox}{left}
\draw[white] (0,0) grid (13,3);

\draw[line width = 2.25, blue] (1.5,1.5*1.73205) -- (1.5,0);
\draw[line width = 2.25, green] (1.5,1.5*1.73205) -- (3,0);
\draw[line width = 2.25, red] (1.5,1.5*1.73205) -- (0,0);

\fill[black] (1.5,1.5*1.73205) circle [radius = 0.15];
\fill[black] (0,0) circle [radius = 0.15];
\fill[black] (3,0) circle [radius = 0.15];
\fill[black] (1.5,0) circle [radius = 0.15];


\draw[line width = 2.25, green] (6.5,1.5*1.73205) -- (5,0);
\draw[line width = 2.25, blue] (6.5,1.5*1.73205) -- (8,0);
\draw[line width = 2.25, red] (6.5,1.5*1.73205) -- (6.5,0);

\fill[black] (6.5,1.5*1.73205) circle [radius = 0.15];
\fill[black] (5,0) circle [radius = 0.15];
\fill[black] (8,0) circle [radius = 0.15];
\fill[black] (6.5,0) circle [radius = 0.15];


\draw[line width = 2.25, blue] (11.5,1.5*1.73205) -- (10,0);
\draw[line width = 2.25, green] (11.5,1.5*1.73205) -- (11.5,0);
\draw[line width = 2.25, red] (11.5,1.5*1.73205) -- (13,0);

\fill[black] (11.5,1.5*1.73205) circle [radius = 0.15];
\fill[black] (10,0) circle [radius = 0.15];
\fill[black] (13,0) circle [radius = 0.15];
\fill[black] (11.5,0) circle [radius = 0.15];

\end{adjustbox}{left}
\end{tikzpicture}
\end{center}

This gives us the matrix:
\begin{center}
\[
\begin{bmatrix}
R & B & G\\
G & R & B\\
B & G & R
\end{bmatrix}
\]
\end{center}

Notice that this rainbow star matrix possesses two special properties:

\begin{enumerate}
    \item Each element within a given row is distinct.
    \item Each element within a given column is distinct. 
\end{enumerate}

The first condition is due to the fact that every edge within a rainbow star must have a different color by definition. The second condition holds because an edge of a given color connecting the same pair of vertices cannot be used in two different rainbow stars (as per the definition of decomposition). 

Thus, such a rainbow star matrix is really just a Latin square of size $n-1$ in disguise. Therefore, the number of such rainbow star matrices is $L_{n-1}$, i.e. the number of size $n-1$ Latin squares. However, we ignore the order of the resulting trees, as this would be rearranging stars in a collection which wouldn't count as a distinct collection. If we were to fix a color onto the diagonal of the rainbow star matrix, (in the above case red), then the real number of collections would be ${L_{n-1}}$ divided by the number of ways to permute the number of rows, $(n-1)!$ Thus, the number of ways to decompose a graph  consisting of $n-1$ monochromatic stars sharing a center into rainbow stars is $\frac{L_{n-1}}{(n-1)!}$. 
\end{proof}

An explicit formula due to Shao and Wei \cite{ShaoWei} for ${L_n}$ is given below.

\begin{equation}
    L_n=n!\sum_{A\epsilon{B_n}}^{}(-1)^{\sigma_0(A)}\binom{{\text{per } A}}{{n}}
\end{equation}

Here, $B_n$ is the set of $n \times n$ matrices with entries in $\{0,1\}$, $\sigma_0(A)$ is the number of zero elements in $A$, and $\text{per } A$ is the permanent of the matrix $A$. Fitting this into our equation $\frac{L_{n-1}}{(n-1)!}$, the $(n-1)!$s cancel out nicely, leaving us with:

\begin{equation}
    \Omega(n) = \sum_{A\epsilon{B_{n-1}}}^{}(-1)^{\sigma_0(A)}\binom{{perA}}{{n-1}}
\end{equation}

Values for $\Omega(n)$ grow very quickly, as shown below. 

\begin{table}[h]
\renewcommand{\arraystretch}{1.5}
\begin{tabular}{|c|c|c|c|c|c|c|c|c|}
\hline
n    & 1 & 2 & 3 & 4 & 5  & 6    & 7       & 8           \\ \hline
$\Omega(n)$ & 1 & 1 & 1 & 2 & 24 & 1344 & 1128960 & 12198297600 \\ \hline

\end{tabular}
\end{table}

\subsection{Extension of \ref{sec:allsame} to Trees}

\begin{thm}
Suppose $G$ has been edge colored so that each color class induces $n-1$ copies of the same monochromatic spanning tree. Then  
Rota's Basis Conjecture holds.
\end{thm}

Our argument stated in section 3.1 isn't restricted to just stars with the same center case. It can be extended to any collection of monochromatic spanning trees that are all the same. This is because for each edge on a rainbow tree, we pick edges from a set of $n-1$ edges on $n-1$ trees. Therefore, the same conversion to a rainbow matrix holds, and the Latin Square argument for identical stars can be used for identical trees (identical except for color, of course). The only difference is that we have to assign each edge to a column in the matrix, giving us a \textbf{rainbow tree matrix}.

We present an example of how our argument applies to a set of identical trees when $n=4$. 
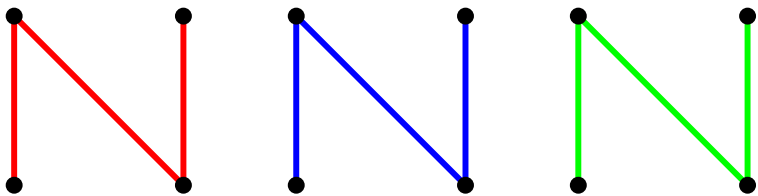
\vspace{1mm}

\begin{center}
\begin{tikzpicture}[scale=0.75] 

\draw[white] (0,0) grid (13,3);

\draw[line width = 2.25, red] (3,0) -- (3,3);
\draw[line width = 2.25, red] (0,3) -- (0,0);
\draw[line width = 2.25, red] (0,3) -- (3,0);


\draw[line width = 2.25, blue] (8,3) -- (8,0);
\draw[line width = 2.25, blue] (5,3) -- (8,0);
\draw[line width = 2.25, blue] (5,3) -- (5,0);


\draw[line width = 2.25, green] (10,0) -- (10,3);
\draw[line width = 2.25, green] (10,3) -- (13,0);
\draw[line width = 2.25, green] (13,3) -- (13,0);

\fill[black] (0,3) circle [radius = 0.15];
\fill[black] (3,3) circle [radius = 0.15];
\fill[black] (0,0) circle [radius = 0.15];
\fill[black] (3,0) circle [radius = 0.15];

\fill[black] (10,3) circle [radius = 0.15];
\fill[black] (13,3) circle [radius = 0.15];
\fill[black] (10,0) circle [radius = 0.15];
\fill[black] (13,0) circle [radius = 0.15];

\fill[black] (5,3) circle [radius = 0.15];
\fill[black] (8,3) circle [radius = 0.15];
\fill[black] (5,0) circle [radius = 0.15];
\fill[black] (8,0) circle [radius = 0.15];

\end{tikzpicture}
\end{center}
\vspace{2mm}
The two resulting collections of rainbow trees would look like this.
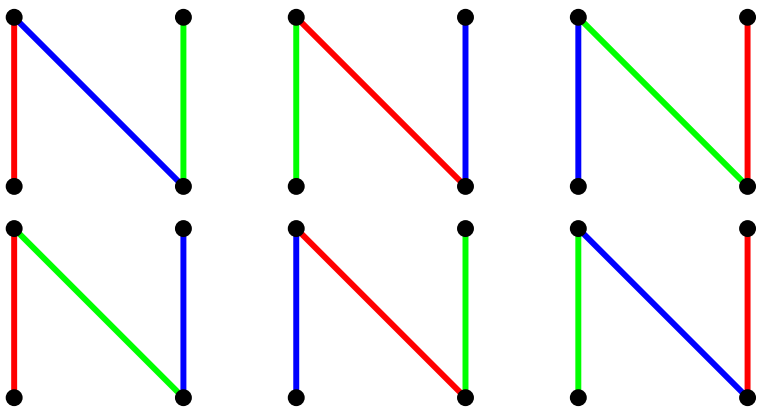
\vspace{2mm}

\begin{center}
\begin{tikzpicture}[scale=0.75] 

\draw[white] (0,0) grid (13,3);

\draw[line width = 2.25, green] (3,0) -- (3,3);
\draw[line width = 2.25, red] (0,3) -- (0,0);
\draw[line width = 2.25, blue] (0,3) -- (3,0);


\draw[line width = 2.25, blue] (8,3) -- (8,0);
\draw[line width = 2.25, red] (5,3) -- (8,0);
\draw[line width = 2.25, green] (5,3) -- (5,0);


\draw[line width = 2.25, blue] (10,0) -- (10,3);
\draw[line width = 2.25, green] (10,3) -- (13,0);
\draw[line width = 2.25, red] (13,3) -- (13,0);

\fill[black] (0,3) circle [radius = 0.15];
\fill[black] (3,3) circle [radius = 0.15];
\fill[black] (0,0) circle [radius = 0.15];
\fill[black] (3,0) circle [radius = 0.15];

\fill[black] (10,3) circle [radius = 0.15];
\fill[black] (13,3) circle [radius = 0.15];
\fill[black] (10,0) circle [radius = 0.15];
\fill[black] (13,0) circle [radius = 0.15];

\fill[black] (5,3) circle [radius = 0.15];
\fill[black] (8,3) circle [radius = 0.15];
\fill[black] (5,0) circle [radius = 0.15];
\fill[black] (8,0) circle [radius = 0.15];

\end{tikzpicture}

\end{center}

\vspace{1mm}

\begin{center}
\begin{tikzpicture}[scale=0.75] 

\draw[white] (0,0) grid (13,3);

\draw[line width = 2.25, blue] (3,0) -- (3,3);
\draw[line width = 2.25, red] (0,3) -- (0,0);
\draw[line width = 2.25, green] (0,3) -- (3,0);


\draw[line width = 2.25, green] (8,3) -- (8,0);
\draw[line width = 2.25, red] (5,3) -- (8,0);
\draw[line width = 2.25, blue] (5,3) -- (5,0);


\draw[line width = 2.25, green] (10,0) -- (10,3);
\draw[line width = 2.25, blue] (10,3) -- (13,0);
\draw[line width = 2.25, red] (13,3) -- (13,0);

\fill[black] (0,3) circle [radius = 0.15];
\fill[black] (3,3) circle [radius = 0.15];
\fill[black] (0,0) circle [radius = 0.15];
\fill[black] (3,0) circle [radius = 0.15];

\fill[black] (10,3) circle [radius = 0.15];
\fill[black] (13,3) circle [radius = 0.15];
\fill[black] (10,0) circle [radius = 0.15];
\fill[black] (13,0) circle [radius = 0.15];

\fill[black] (5,3) circle [radius = 0.15];
\fill[black] (8,3) circle [radius = 0.15];
\fill[black] (5,0) circle [radius = 0.15];
\fill[black] (8,0) circle [radius = 0.15];

\end{tikzpicture}
\end{center}
\vspace{2mm}
For the matrix, we let the first column be the left edge, the middle column to be the diagonal edge, and the right column to be the right edge. The corresponding rainbow tree matrices would be the following, where the left matrix represents the top collection and the right matrix represents the bottom matrix.

\begin{center}
\[
\begin{bmatrix}
R & B & G\\
G & R & B\\
B & G & R
\end{bmatrix}
\&
\begin{bmatrix}
R & G & B\\
B & R & G\\
G & B & R
\end{bmatrix}
\]
\end{center}
\vspace{2mm}
This method of counting the number of collections will hold true as long as all monochromatic spanning trees within a collection are identical.

\subsection{Stars-to-stars Can Fail}  \label{sec:counter}
We next consider a situation where monochromatic stars do not lead to rainbow stars.

Consider a graph $G$ on $4$ vertices, which is edge-colored in such a way that two of the induced monochromatic stars share a center. Below, we see that the red and blue stars share the top left corner as their center. We also run into the problem where we can't color our vertex according to its star, so we'll just mix the colors of the stars with internal vertices at that spot, conveniently red and blue make purple. We will mark the bottom left corner with a black vertex.
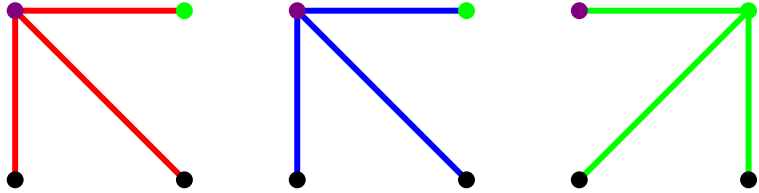
\vspace{2mm}
\begin{center}
\begin{tikzpicture}[scale=0.75] 
\begin{adjustbox}{left}
\draw[white] (0,0) grid (11,3);

\draw[line width = 2.25, red] (0,3) -- (3,3);
\draw[line width = 2.25, red] (0,3) -- (0,0);
\draw[line width = 2.25, red] (0,3) -- (3,0);


\draw[line width = 2.25, blue] (5,3) -- (8,3);
\draw[line width = 2.25, blue] (5,3) -- (8,0);
\draw[line width = 2.25, blue] (5,3) -- (5,0);


\draw[line width = 2.25, green] (13,3) -- (10,3);
\draw[line width = 2.25, green] (13,3) -- (10,0);
\draw[line width = 2.25, green] (13,3) -- (13,0);

\fill[violet] (0,3) circle [radius = 0.15];
\fill[green] (3,3) circle [radius = 0.15];
\fill[black] (0,0) circle [radius = 0.15];
\fill[black] (3,0) circle [radius = 0.15];

\fill[violet] (10,3) circle [radius = 0.15];
\fill[green] (13,3) circle [radius = 0.15];
\fill[black] (10,0) circle [radius = 0.15];
\fill[black] (13,0) circle [radius = 0.15];

\fill[violet] (5,3) circle [radius = 0.15];
\fill[green] (8,3) circle [radius = 0.15];
\fill[black] (5,0) circle [radius = 0.15];
\fill[black] (8,0) circle [radius = 0.15];
\end{adjustbox}{left}
\end{tikzpicture}
\end{center}
\vspace{2mm}
There is only one possible way to decompose $G$ into 3 rainbow trees, ignoring symmetry between the red and blue edges: 
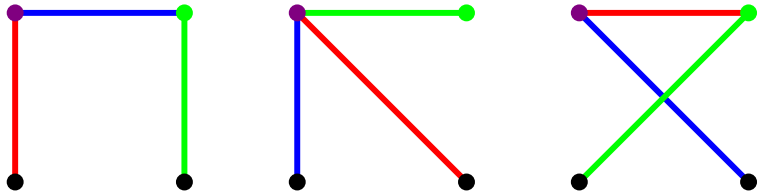
\vspace{2mm}
\begin{center}
\begin{tikzpicture}[scale=0.75] 
\begin{adjustbox}{left}
\draw[white] (0,0) grid (13,3);

\draw[line width = 2.25, red] (0,3) -- (0,0);
\draw[line width = 2.25, red] (5,3) -- (8,0);
\draw[line width = 2.25, red] (10,3) -- (13,3);


\draw[line width = 2.25, blue] (0,3) -- (3,3);
\draw[line width = 2.25, blue] (5,3) -- (5,0);
\draw[line width = 2.25, blue] (10,3) -- (13,0);


\draw[line width = 2.25, green] (3,3) -- (3,0);
\draw[line width = 2.25, green] (5,3) -- (8,3);
\draw[line width = 2.25, green] (10,0) -- (13,3);

\fill[violet] (0,3) circle [radius = 0.15];
\fill[green] (3,3) circle [radius = 0.15];
\fill[black] (0,0) circle [radius = 0.15];
\fill[black] (3,0) circle [radius = 0.15];

\fill[violet] (10,3) circle [radius = 0.15];
\fill[green] (13,3) circle [radius = 0.15];
\fill[black] (10,0) circle [radius = 0.15];
\fill[black] (13,0) circle [radius = 0.15];

\fill[violet] (5,3) circle [radius = 0.15];
\fill[green] (8,3) circle [radius = 0.15];
\fill[black] (5,0) circle [radius = 0.15];
\fill[black] (8,0) circle [radius = 0.15];
\end{adjustbox}{left}
\end{tikzpicture}
\end{center}
\vspace{2mm}

Notice that only one of the trees above is a rainbow star. Thus, it is impossible to decompose $G$ into rainbow stars. 

\subsection{All Stars Have a Center that Falls on One of Two Nodes}

Although an arbitrary arrangement of monochromatic stars does not necessarily lead to a collection of rainbow stars, we can give another proof of Rota's Conjecture in a special case.

\begin{thm}\label{thm:samecenter}
    Suppose a graph $G$ on $n$ vertices is edge colored such that it consists of $n-1$ monochromatic spanning stars that have a center on one of two vertices. Then, it is possible to decompose $G$ into $n-1$ disjoint spanning rainbow trees. 
\end{thm}

\begin{proof}
Let's suppose we have a graph $G$ with $n_1$ monochromatic stars centered on vertex $1$ and $n_2$ monochromatic stars centered on vertex $2$, where $n_1 + n_2 = n-1$. We want to decompose such a graph into $n-1$ rainbow trees. We will construct these $n-1$ rainbow trees from scratch. 

First, some notation: let $C_1,C_2$ denote the set of colors that belong to $1$ and $2$, respectively. Then, $\left| C_1 \right| = n_1$ and $\left| C_2 \right| = n_2$. 
For the vertices that serve as the centers of stars, we define an arbitrary 'order' of colors. For instance, 'for vertex $k$, blue comes clockwise after red'. We will also reorder the colors within $C_k$ and $C_j$ to reflect this arbitrary choice of color ordering. Then, construct $n-1$ trees using the following idea: 
\begin{itemize}
    \item For each tree, we choose the edge $e_{jk}$, where $c$ is distinct for each tree. 
    \item Pick any one of the trees-under-construction. Say that for one given tree, the edge connecting $k$ and $j$ is of the form $e_{jk}$, where $e_{jk} \in S_j$. Then, for the $n_1-1$ remaining colors in $C_j$, we place edges of the form $e_{j\delta}$, where $\delta = k+i$ and $e_{j\delta} \in S_i$; this way, the order of the colors of the edges emanating from $j$ matches the order that was arbitrarily determined previously. 
    \item After that, we will be left with $(n-1)-(n_1)=n_2$ points that are still not connected to a vertex. To rescue those points, we will connect them each to $k$; the order of the colors of the edges will follow the pre-determined order. 
\end{itemize}
Thus, we see a 'rotation' motif present throughout such constructs. 

For instance, say we have a graph on $n=5$ vertices that consists of the following monochromatic stars. 


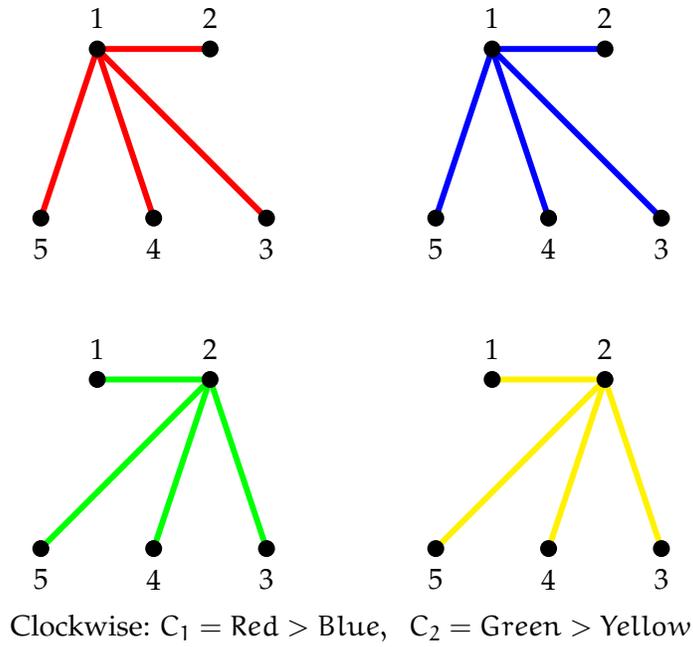
\begin{center}
\begin{tikzpicture}[scale=0.75] 
\begin{adjustbox}{left}
\draw[white] (0,0) grid (11,3);

\draw[line width = 2.25, red] (0,0) -- (1,3) ;
\draw[line width = 2.25, red] (2,0) -- (1,3);
\draw[line width = 2.25, red] (4,0) -- (1,3);
\draw[line width = 2.25, red] (3,3) -- (1,3);

\fill[black] (0,0) circle [radius = 0.15] node[below=4pt] {5};
\fill[black] (2,0) circle [radius = 0.15] node[below=4pt] {4};
\fill[black] (4,0) circle [radius = 0.15] node[below=4pt] {3};
\fill[black] (1,3) circle [radius = 0.15] node[above=4pt] {1};
\fill[black] (3,3) circle [radius = 0.15] node[above=4pt] {2};

\draw[line width = 2.25, blue] (7,0) -- (8,3);
\draw[line width = 2.25, blue] (9,0) -- (8,3);
\draw[line width = 2.25, blue] (11,0) -- (8,3);
\draw[line width = 2.25, blue] (10,3) -- (8,3);

\fill[black] (7,0) circle [radius = 0.15] node[below=4pt] {5};
\fill[black] (9,0) circle [radius = 0.15] node[below=4pt] {4};
\fill[black] (11,0) circle [radius = 0.15] node[below=4pt] {3};
\fill[black] (8,3) circle [radius = 0.15] node[above=4pt] {1};
\fill[black] (10,3) circle [radius = 0.15] node[above=4pt] {2};

\end{adjustbox}{left}
\end{tikzpicture}
\end{center}
\begin{center}
\begin{tikzpicture}[scale=0.75] 
\begin{adjustbox}{left}
\draw[white] (0,0) grid (11,5);

\draw[line width = 2.25, green] (0,2) -- (3,5);
\draw[line width = 2.25, green] (2,2) -- (3,5);
\draw[line width = 2.25, green] (4,2) -- (3,5);
\draw[line width = 2.25, green] (3,5) -- (1,5);

\fill[black] (0,2) circle [radius = 0.15] node[below=4pt] {5};
\fill[black] (2,2) circle [radius = 0.15] node[below=4pt] {4};
\fill[black] (4,2) circle [radius = 0.15] node[below=4pt] {3};
\fill[black] (1,5) circle [radius = 0.15] node[above=4pt] {1};
\fill[black] (3,5) circle [radius = 0.15] node[above=4pt] {2};


\draw[line width = 2.25, yellow] (7,2) -- (10,5);
\draw[line width = 2.25, yellow] (9,2) -- (10,5);
\draw[line width = 2.25, yellow] (11,2) -- (10,5);
\draw[line width = 2.25, yellow] (10,5) -- (8,5);

\fill[black] (7,2) circle [radius = 0.15] node[below=4pt] {5};
\fill[black] (9,2) circle [radius = 0.15] node[below=4pt] {4};
\fill[black] (11,2) circle [radius = 0.15] node[below=4pt] {3};
\fill[black] (8,5) circle [radius = 0.15] node[above=4pt] {1};
\fill[black] (10,5) circle [radius = 0.15] node[above=4pt] {2};
\fill[white] (5.5,0) circle [radius = 0.15] node[above=4pt, black] {{Clockwise: $C_1 = Red > Blue, \; \; C_2 = Green > Yellow$}};

\end{adjustbox}{left}
\end{tikzpicture}

\end{center}


Then, we can start by laying down the $k-j$ connections.

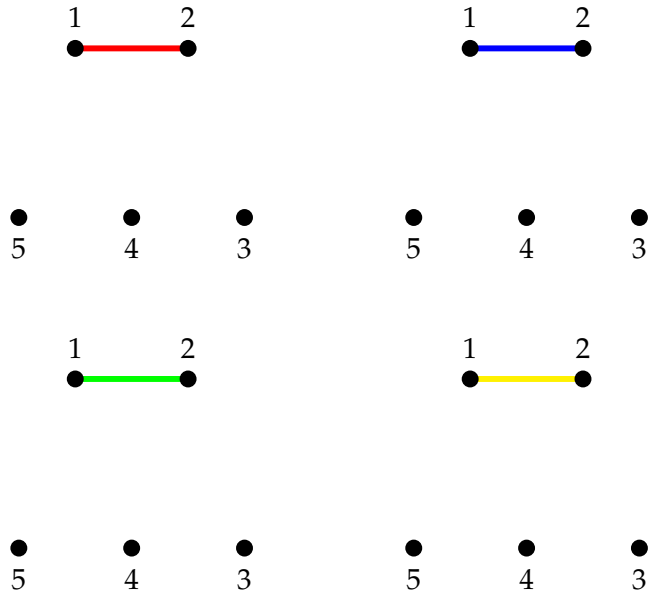
\begin{center}
\begin{tikzpicture}[scale=0.75] 
\begin{adjustbox}{left}
\draw[white] (0,0) grid (11,3);

\draw[line width = 2.25, red] (3,3) -- (1,3);

\fill[black] (0,0) circle [radius = 0.15] node[below=4pt] {5};
\fill[black] (2,0) circle [radius = 0.15] node[below=4pt] {4};
\fill[black] (4,0) circle [radius = 0.15] node[below=4pt] {3};
\fill[black] (1,3) circle [radius = 0.15] node[above=4pt] {1};
\fill[black] (3,3) circle [radius = 0.15] node[above=4pt] {2};


\draw[line width = 2.25, blue] (10,3) -- (8,3);

\fill[black] (7,0) circle [radius = 0.15] node[below=4pt] {5};
\fill[black] (9,0) circle [radius = 0.15] node[below=4pt] {4};
\fill[black] (11,0) circle [radius = 0.15] node[below=4pt] {3};
\fill[black] (8,3) circle [radius = 0.15] node[above=4pt] {1};
\fill[black] (10,3) circle [radius = 0.15] node[above=4pt] {2};

\end{adjustbox}{left}
\end{tikzpicture}
\end{center}
\begin{center}
\begin{tikzpicture}[scale=0.75] 
\begin{adjustbox}{left}
\draw[white] (0,0) grid (11,3);

\draw[line width = 2.25, green] (3,3) -- (1,3);

\fill[black] (0,0) circle [radius = 0.15] node[below=4pt] {5};
\fill[black] (2,0) circle [radius = 0.15] node[below=4pt] {4};
\fill[black] (4,0) circle [radius = 0.15] node[below=4pt] {3};
\fill[black] (1,3) circle [radius = 0.15] node[above=4pt] {1};
\fill[black] (3,3) circle [radius = 0.15] node[above=4pt] {2};


\draw[line width = 2.25, yellow] (10,3) -- (8,3);

\fill[black] (7,0) circle [radius = 0.15] node[below=4pt] {5};
\fill[black] (9,0) circle [radius = 0.15] node[below=4pt] {4};
\fill[black] (11,0) circle [radius = 0.15] node[below=4pt] {3};
\fill[black] (8,3) circle [radius = 0.15] node[above=4pt] {1};
\fill[black] (10,3) circle [radius = 0.15] node[above=4pt] {2};

\end{adjustbox}{left}
\end{tikzpicture}
\end{center}

If the edge connecting $k$ to $j$ is color $c$ and $c$ belongs to $k$, we connect points to $k$ with a specific color depending on the pre-determined color ordering with respect to $c$. 

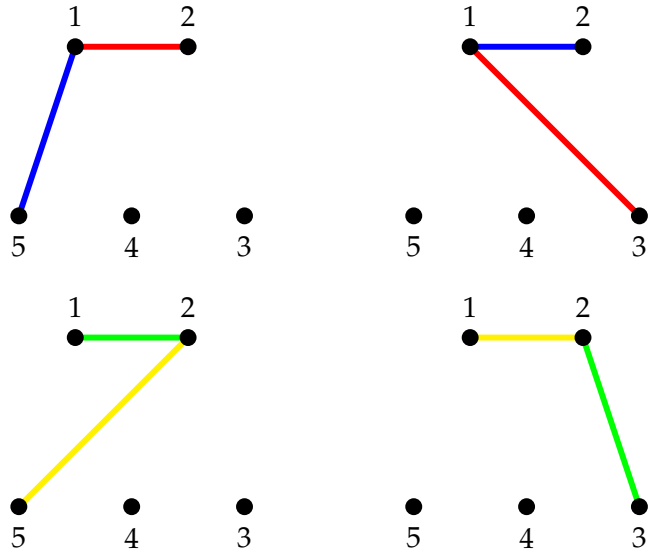
\begin{center}
\begin{tikzpicture}[scale=0.75] 
\begin{adjustbox}{left}
\draw[white] (0,0) grid (11,3);

\draw[line width = 2.25, red] (3,3) -- (1,3);
\draw[line width = 2.25, blue] (0,0) -- (1,3);

\fill[black] (0,0) circle [radius = 0.15] node[below=4pt] {5};
\fill[black] (2,0) circle [radius = 0.15] node[below=4pt] {4};
\fill[black] (4,0) circle [radius = 0.15] node[below=4pt] {3};
\fill[black] (1,3) circle [radius = 0.15] node[above=4pt] {1};
\fill[black] (3,3) circle [radius = 0.15] node[above=4pt] {2};


\draw[line width = 2.25, red] (8,3) -- (11,0);
\draw[line width = 2.25, blue] (10,3) -- (8,3);

\fill[black] (7,0) circle [radius = 0.15] node[below=4pt] {5};
\fill[black] (9,0) circle [radius = 0.15] node[below=4pt] {4};
\fill[black] (11,0) circle [radius = 0.15] node[below=4pt] {3};
\fill[black] (8,3) circle [radius = 0.15] node[above=4pt] {1};
\fill[black] (10,3) circle [radius = 0.15] node[above=4pt] {2};

\end{adjustbox}{left}
\end{tikzpicture}
\end{center}
\begin{center}
\begin{tikzpicture}[scale=0.75] 
\begin{adjustbox}{left}
\draw[white] (0,0) grid (11,3);

\draw[line width = 2.25, yellow] (0,0) -- (3,3);
\draw[line width = 2.25, green] (3,3) -- (1,3);

\fill[black] (0,0) circle [radius = 0.15] node[below=4pt] {5};
\fill[black] (2,0) circle [radius = 0.15] node[below=4pt] {4};
\fill[black] (4,0) circle [radius = 0.15] node[below=4pt] {3};
\fill[black] (1,3) circle [radius = 0.15] node[above=4pt] {1};
\fill[black] (3,3) circle [radius = 0.15] node[above=4pt] {2};


\draw[line width = 2.25, green] (11,0) -- (10,3);
\draw[line width = 2.25, yellow] (10,3) -- (8,3);

\fill[black] (7,0) circle [radius = 0.15] node[below=4pt] {5};
\fill[black] (9,0) circle [radius = 0.15] node[below=4pt] {4};
\fill[black] (11,0) circle [radius = 0.15] node[below=4pt] {3};
\fill[black] (8,3) circle [radius = 0.15] node[above=4pt] {1};
\fill[black] (10,3) circle [radius = 0.15] node[above=4pt] {2};

\end{adjustbox}{left}
\end{tikzpicture}
\end{center}


Finally, we follow a similar process for $j$ but with the remaining vertices and follow the clockwise pre-determined ordering.

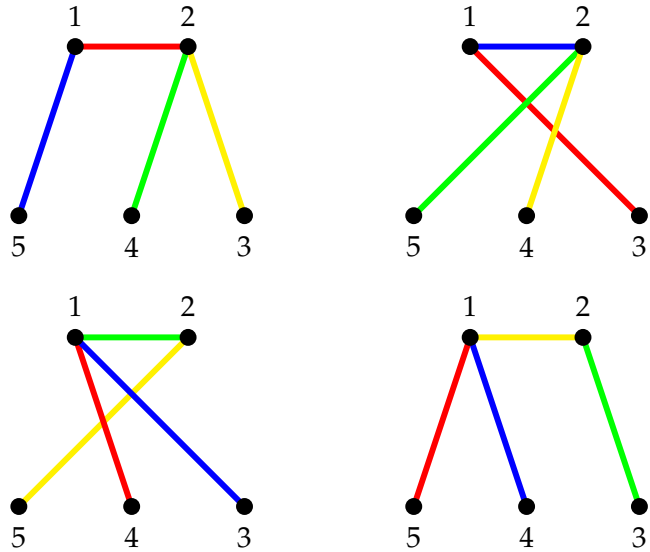
\begin{center}
\begin{tikzpicture}[scale=0.75] 
\begin{adjustbox}{left}
\draw[white] (0,0) grid (11,3);

\draw[line width = 2.25, green] (3,3) -- (2,0);
\draw[line width = 2.25, yellow] (3,3) -- (4,0);
\draw[line width = 2.25, blue] (0,0) -- (1,3);
\draw[line width = 2.25, red] (3,3) -- (1,3);

\fill[black] (0,0) circle [radius = 0.15] node[below=4pt] {5};
\fill[black] (2,0) circle [radius = 0.15] node[below=4pt] {4};
\fill[black] (4,0) circle [radius = 0.15] node[below=4pt] {3};
\fill[black] (1,3) circle [radius = 0.15] node[above=4pt] {1};
\fill[black] (3,3) circle [radius = 0.15] node[above=4pt] {2};


\draw[line width = 2.25, red] (11,0) -- (8,3);
\draw[line width = 2.25, green] (7,0) -- (10,3);
\draw[line width = 2.25, yellow] (9,0) -- (10,3);
\draw[line width = 2.25, blue] (10,3) -- (8,3);

\fill[black] (7,0) circle [radius = 0.15] node[below=4pt] {5};
\fill[black] (9,0) circle [radius = 0.15] node[below=4pt] {4};
\fill[black] (11,0) circle [radius = 0.15] node[below=4pt] {3};
\fill[black] (8,3) circle [radius = 0.15] node[above=4pt] {1};
\fill[black] (10,3) circle [radius = 0.15] node[above=4pt] {2};

\end{adjustbox}{left}
\end{tikzpicture}
\end{center}
\begin{center}
\begin{tikzpicture}[scale=0.75] 
\begin{adjustbox}{left}
\draw[white] (0,0) grid (11,3);

\draw[line width = 2.25, yellow] (0,0) -- (3,3);
\draw[line width = 2.25, red] (2,0) -- (1,3);
\draw[line width = 2.25, blue] (4,0) -- (1,3);
\draw[line width = 2.25, green] (3,3) -- (1,3);

\fill[black] (0,0) circle [radius = 0.15] node[below=4pt] {5};
\fill[black] (2,0) circle [radius = 0.15] node[below=4pt] {4};
\fill[black] (4,0) circle [radius = 0.15] node[below=4pt] {3};
\fill[black] (1,3) circle [radius = 0.15] node[above=4pt] {1};
\fill[black] (3,3) circle [radius = 0.15] node[above=4pt] {2};


\draw[line width = 2.25, blue] (9,0) -- (8,3);
\draw[line width = 2.25, red] (7,0) -- (8,3);
\draw[line width = 2.25, green] (11,0) -- (10,3);
\draw[line width = 2.25, yellow] (10,3) -- (8,3);

\fill[black] (7,0) circle [radius = 0.15] node[below=4pt] {5};
\fill[black] (9,0) circle [radius = 0.15] node[below=4pt] {4};
\fill[black] (11,0) circle [radius = 0.15] node[below=4pt] {3};
\fill[black] (8,3) circle [radius = 0.15] node[above=4pt] {1};
\fill[black] (10,3) circle [radius = 0.15] node[above=4pt] {2};

\end{adjustbox}{left}
\end{tikzpicture}
\end{center}

To see why this idea does not produce any cycles or disconnections, note the following points:
\begin{itemize}
    \item Each vertex $m \neq j,k$ is connected to \textbf{exactly one} of $j$ or $k$ (either $m$ was connected to $j$ through the second step, or it was one of the remaining $n_2$ vertices that was connected to $k$ through the third step). Also, every tree, by construction, has an edge connecting $j$ to $k$. Therefore, there are no disconnected points in our constructed rainbow trees. 
    \item Because there are no repeated edges or edges of the form $e_{mr}$ in the resulting rainbow trees, where $m,r \neq k,j$, there are no cycles.
\end{itemize}

We can also guarantee that the collection of rainbow trees is disjoint; if one were given the original set of monochromatic trees and the orderings of the stars within a center, then it would be possible to tell which rainbow tree an edge would belong to. This is possible because for a center, we can place the remaining edges of the other colors emanating from that center based off of the ordering of stars on that center, and then place the remaining edges of the stars with the other center based off of the ordering of stars on that center. Being able to dictate which tree an edge is in guarantees disjointedness.

Thus, we can decompose $G$ into $n-1$ disjoint spanning rainbow trees. 
\end{proof}

\subsection{When Stars-to-Stars Holds}
Sections 3.1 and 3.2 suggest that it is possible to make rainbow stars from monochromatic stars with different or same centers, respectively, but a question to ponder is if these are the only cases of monochromatic stars which produce rainbow stars. A natural extension to the star problem addressed above is that of whether a graph induced by monochromatic spanning paths can be decomposed into a set of rainbow spanning paths. 

\begin{thm}\label{thm:main}
    Stars-to-stars holds if and only if all monochromatic stars share the same center or all have different centers. 
\end{thm}
\begin{proof}
The `if' implication has been established in Theorem \ref{thm:starsdifferent} and Theorem \ref{thm:starssame}.
We are left to prove the converse.

The statement is trivial for $n=2$ because there is only $1$ star consisting of $1$ edge. 
For $n \ge 3$, say we start out with a decomposition of a graph into monochromatic stars. In said decomposition, say there are $s_k$ stars on vertex $k$. Then the degree of vertex $k$ is $$D_k = s_k(n-1) + (n-1)-s_k = (s_k+1)(n-2)+1.$$
The $s_k(n-1)$ comes from the number of stars on that center times the number of edges per star, and the $(n-1)-s_k$ comes from the number of stars not on that center times $1$ for the number of edges that connect to vertex $k$ per star. The right-hand side is a factored version of the equation.

Now, consider a hypothetical decomposition of the same graph into rainbow stars. There must be exactly $s_k$ rainbow stars centered at vertex $k$. If there were more than $s_k$ rainbow stars on vertex $k$, we would have at least $(s_k+2)(n-2)+1$ edges incident to vertex $k$, which is larger than $D_k$. Likewise, if there were less than $s_k$ rainbow stars on vertex $k$, we would have at most $(s_k)(n-2)+1$ edges incident to vertex $k$, which is smaller than $D_k$.

Now, note that if a monochromatic star is centered at $k$, then all edges of that color emanate from vertex $k$. If that monochromatic star is not centered at $k$, then only one edge of that color connects to vertex $k$. Since each rainbow star must have exactly one of a certain color, exactly $0$, $1$, or $n-1$ rainbow stars must be centered at vertex $k$. There being $0$ stars on a vertex must be included because there are $n$ vertices and $n-1$ stars, and by the Pigeonhole Principle, at least one vertex does not have a star centered at it. Therefore, $s_k = 0,1,n-1$, proving our theorem. Note, the $1$ or $n-1$ rainbow stars being centered at a point corresponds with Theorems \ref{thm:starsdifferent} and \ref{thm:starssame}, respectively.
\end{proof}

\section {Paths-to-Paths}\label{section:paths}

\begin{defn}
    A \textbf{spanning path} on a set of $n$ vertices is a set of $n-1$ edges such that 
    \begin{enumerate}[(a)]
        \item All vertices are included in the path (hence the name ``spanning"), and 
        \item For a given spanning path, all vertices have degree at most 2.
    \end{enumerate}
    In other words, a spanning path can be drawn in one motion without lifting the pencil off of the paper. 
\end{defn}

\begin{defn}
Suppose $G$ is a graph on $n$ vertices that has been edge colored with $n-1$ colors so that each monochromatic subgraph induces a spanning path. Then we say \textbf{Paths-to-Paths}
 holds for $G$ if one can find a collection of $n-1$ disjoint spanning rainbow paths.
\end{defn}



We next show that Paths-to-Paths is in general not possible. 

\begin{prop} \label{prop:path}
  For the edge coloring of the graph $G$ on $4$ vertices depicted below, there does not exist any collection of $3$ disjoint rainbow paths.
\end{prop}

\begin{proof}
Suppose $G$ is the edge colored graph on $4$ nodes whose color classes induce the 3 monochromatic spanning paths shown below. 
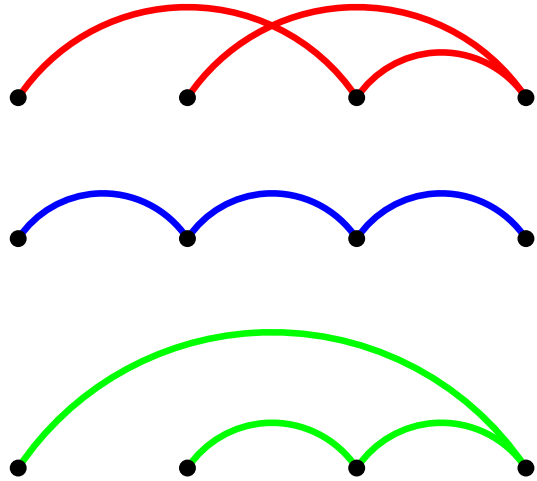
\vspace{4mm}
\begin{center}
\begin{tikzpicture}[scale=0.75] 
\begin{adjustbox}{left}

\draw[white] (0,0) grid (9,2);

\draw[line width = 2.5, red] (9,1) arc (33.690067:180-33.690067:1.8027756);
\draw[line width = 2.5, red] (6,1) arc (33.6900675:180-33.6900675:3.605551275);
\draw[line width = 2.5, red] (9,1) arc (33.6900675:180-33.6900675:3.605551275);

\fill[black] (0,1) circle [radius = 0.15];
\fill[black] (3,1) circle [radius = 0.15];
\fill[black] (6,1) circle [radius = 0.15];
\fill[black] (9,1) circle [radius = 0.15];

\end{adjustbox}
\end{tikzpicture}
\end{center}

\begin{center}
\begin{tikzpicture}[scale=0.75] 
\begin{adjustbox}{left}

\draw[white] (0,0) grid (9,2);

\draw[line width = 2.5, blue] (3,1) arc (33.690067:180-33.690067:1.8027756);
\draw[line width = 2.5, blue] (6,1) arc (33.690067:180-33.690067:1.8027756);
\draw[line width = 2.5, blue] (9,1) arc (33.690067:180-33.690067:1.8027756);
\fill[black] (0,1) circle [radius = 0.15];
\fill[black] (3,1) circle [radius = 0.15];
\fill[black] (6,1) circle [radius = 0.15];
\fill[black] (9,1) circle [radius = 0.15];

\end{adjustbox}
\end{tikzpicture}
\end{center}

\begin{center}
\begin{tikzpicture}[scale=0.75] 
\begin{adjustbox}{left}

\draw[white] (0,0) grid (9,2);

\draw[line width = 2.5, green] (9,1) arc (33.690067:180-33.690067:3*1.8027756);
\draw[line width = 2.5, green] (6,1) arc (33.690067:180-33.690067:1.8027756);
\draw[line width = 2.5, green] (9,1) arc (33.6900675:180-33.6900675:1.8027756);
\fill[black] (0,1) circle [radius = 0.15];
\fill[black] (3,1) circle [radius = 0.15];
\fill[black] (6,1) circle [radius = 0.15];
\fill[black] (9,1) circle [radius = 0.15];

\end{adjustbox}
\end{tikzpicture}
\end{center}

Notice that the first vertex has only 3 total edges emanating from it (1 red, 1 blue, and 1 green). Thus, each of the 3 rainbow spanning paths must have the first vertex be either a starting or ending vertex; without loss of generality, we assume the first vertex to be the starting node for all 3 of the rainbow spanning paths. 

Now, consider constructing the rainbow spanning path that uses the red edge between the first node and the third node. 
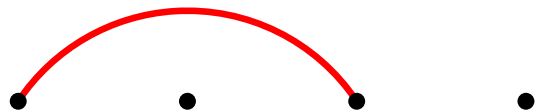
\vspace{4mm}
\begin{center}
\begin{tikzpicture}[scale=0.75] 
\begin{adjustbox}{left}

\draw[white] (0,0) grid (9,2);

\draw[line width = 2.5, red] (6,1) arc (33.690067:180-33.690067:2*1.8027756);
\fill[black] (0,1) circle [radius = 0.15];
\fill[black] (3,1) circle [radius = 0.15];
\fill[black] (6,1) circle [radius = 0.15];
\fill[black] (9,1) circle [radius = 0.15];

\end{adjustbox}
\end{tikzpicture}
\end{center}

The remaining edges in the rainbow spanning path must be either blue or green. Thus, the two possibilities for the rainbow spanning path are as presented below. 
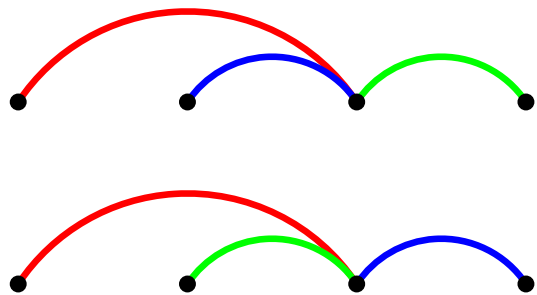
\vspace{4mm}
\begin{center}
\begin{tikzpicture}[scale=0.75] 
\begin{adjustbox}{left}

\draw[white] (0,0) grid (9,2);

\draw[line width = 2.5, red] (6,1) arc (33.690067:180-33.690067:2*1.8027756);
\draw[line width = 2.5, blue] (6,1) arc (33.690067:180-33.690067:1.8027756);
\draw[line width = 2.5, green] (9,1) arc (33.6900675:180-33.6900675:1.8027756);
\fill[black] (0,1) circle [radius = 0.15];
\fill[black] (3,1) circle [radius = 0.15];
\fill[black] (6,1) circle [radius = 0.15];
\fill[black] (9,1) circle [radius = 0.15];

\end{adjustbox}
\end{tikzpicture}
\end{center}

\begin{center}
\begin{tikzpicture}[scale=0.75] 
\begin{adjustbox}{left}

\draw[white] (0,0) grid (9,2);

\draw[line width = 2.5, red] (6,1) arc (33.690067:180-33.690067:2*1.8027756);
\draw[line width = 2.5, blue] (9,1) arc (33.690067:180-33.690067:1.8027756);
\draw[line width = 2.5, green] (6,1) arc (33.6900675:180-33.6900675:1.8027756);
\fill[black] (0,1) circle [radius = 0.15];
\fill[black] (3,1) circle [radius = 0.15];
\fill[black] (6,1) circle [radius = 0.15];
\fill[black] (9,1) circle [radius = 0.15];

\end{adjustbox}
\end{tikzpicture}
\end{center}

However, notice that neither of the two options presented above is a path; both of them are stars. Therefore, the original graph cannot be decomposed into 3 rainbow spanning paths, establishing the claim. 
\end{proof}

\section{Further thoughts}

We end with some discussion and ideas for future research. Further work regarding Conjecture \ref{conj:main} would include considering the following questions.
\begin{itemize}

    \item If the colors classes induce monochromatic spanning paths, under which circumstances can we find spanning rainbow \emph{paths} (i.e. when does Paths-to-Paths hold)?
    \smallskip
    
    \item Can we prove Conjecture \ref{conj:main} for graphs where the color classes induce monochromatic spanning \emph{paths} (i.e. does Paths-to-Trees hold)?

    \end{itemize}

We also remark that Rota's Basis Conjecture makes sense in the more general context of matroid theory. Recall that a \emph{matroid} is a pair $(M,E)$ consisting of a finite ground set $E$ and a set $M$ of subsets of $E$, called \emph{independent sets}, satisfying:
\begin{enumerate}[(a)]
\item $\emptyset \in M$;
\item If $X \in M$ and $Y \subset X$ then $Y \in M$;
\item If  $X, Y \in M$ and $|X| > |Y|$, then there exists an element $x \in X \setminus Y$ so that $Y \cup \{x\} \in B$. 
\end{enumerate}
The maximal elements of $M$ are called the \emph{bases} of $M$. The \emph{rank} of $M$ is the cardinality of any (and hence all) basis element. Rota's Conjecture can then be stated in full generality as follows.

\begin{conj}\label{conj:full}
Suppose $M$ is a rank $n$ matroid with a collection of disjoint bases ${\mathcal B} = \{B_1, \dots, B_{n}\}$. Then one can arrange the elements of ${\mathcal B}$ into an $n \times n$ matrix in such a way that the $i$th row consists of the elements of $B_i$ and each column is also a basis of $M$. 
\end{conj}

In other words, if we think of each $B_i$ as a color class, we obtain a collection of \emph{colorful} bases ${\mathcal C} = \{C_1, \dots, C_n\}$ where each $C_i$ contains exactly one element from each $B_j \in {\mathcal B}$.
In this paper, we considered matroids coming from a graph $G = (V,E)$, where the ground set is given by the set $E$ of edges, and independent sets are given by collections of edges that do contain a cycle.

Conjecture \ref{conj:full} has been established for
all paving matroids by Geelen and Humphries in \cite{GeeHum}, and for matroids of rank at most $3$ by Chan \cite{Chan}.  It was shown by Wild \cite{Wild} that a stronger version of the conjecture holds for \emph{strongly base orderable} matroids. This class is closed under duality and taking minors, and includes all gammoids.  

In \cite{BKPS} it was shown that for any matroid one can always find $(1/2 - \text{o}(1))n$ disjoint transversal bases (so that we are `halfway' to Rota's Conjecture). Recently, Rota's Basis Conjecture for realizable matrices over finite fields has been established in a probabilistic setting by Sauermann \cite{Sauermann}. Interest in Rota's Basis Conjecture continues to this day, as evidenced by the recent Polymath project \cite{Polymath}. It would be interesting to consider Rota's Basis conjecture for other classes of matroids.

\section*{Acknowledgements}
We would like to thank the 2022 Mathworks Honors Summer Math Camp for their support and encouragement. We would like to extend our thanks to Dr. Max Warshauer and Dr. Eugene Curtin of Texas State University for connecting us after said camp. We would finally like to thank Dr. Anton Dochtermann of Texas State University for continually supporting and guiding us throughout the research project and providing a great working environment.

\vspace{20mm}
\contacts

\end{document}